\documentclass[11pt]{article}
\usepackage{epsfig,amsmath,amsthm, wasysym, xcolor}
\usepackage{algorithm,algorithmic}
\usepackage{graphicx, tikz} 
\usetikzlibrary{fit, shapes}
\usetikzlibrary{positioning, calc}
\usepackage{bbm, dsfont}
\usepackage{parskip}

\usepackage[hidelinks]{hyperref}

\usepackage[left=1in, right=1in, top=1.5in, bottom=1in]{geometry}

\newtheorem{df}{Definition}[section]
\newtheorem{theorem}[df]{Theorem}
\newtheorem{lemma}[df]{Lemma}
\newtheorem{corollary}[df]{Corollary}
\newtheorem{proposition}[df]{Proposition}

\usepackage{tcolorbox}

\title{Breaking the Symmetries of Amenable Graphs}
\author{Christine T. Cheng \\Department of Electrical Engineering and Computer Science \\ University of Wisconsin-Milwaukee}
\date{\today}

\begin{document}


\begingroup
\setlength{\parskip}{0pt}
\maketitle
\endgroup

\begin{abstract} 
 In this paper, we consider two ways of breaking a graph's symmetry:  {\it distinguishing labelings} and {\it fixing sets}.  A {\it distinguishing labeling} $\phi$ of $G$ colors the vertices of $G$ so that the only automorphism of the labeled graph $(G, \phi)$ is the identity map.  The {\it distinguishing number} of $G$, $D(G)$, is the fewest number of colors needed to create a distinguishing labeling of $G$.   A subset $S$ of vertices is a {\it fixing set} of $G$ if the only automorphism of $G$ that fixes every element in $S$ is the identity map. The {\it fixing number} of $G$, $Fix(G)$, is the size of a smallest fixing set.    A fixing set $S$ of $G$ can  be translated into  a distinguishing labeling $\phi_S$ by assigning distinct colors to the vertices in $S$ and assigning another color (e.g., the ``null" color) to the vertices not in $S$.  
 
 Color refinement is a well-known efficient heuristic for graph isomorphism.  A graph  $G$ is {\it amenable}  if, for any graph $H$, color refinement correctly determines whether $G$ and $H$ are isomorphic or not.   Using the characterization of amenable graphs by Arvind et al.~as a starting point, we show that both $D(G)$ and $Fix(G)$  can be computed in  $O((|V(G)|+|E(G)|) \log |V(G)|)$ time when $G$ is an amenable graph.
 \end{abstract}

\section{Introduction}
\label{sec:intro}

	A typical graph $G$ has many symmetries. 
It can be difficult if not impossible to tell one vertex from another.  A natural thing to do is to label or color  the  vertices of $G$ so that  the vertices can be distinguished from each other.   One option is to just assign distinct colors to the vertices.  But a more interesting option is to take into account the structure of $G$ and use as few colors as possible to break $G$'s symmetries.   We consider two ways of doing this: {\it distinguishing labelings} \cite{AlCo96} and {\it fixing sets} \cite{DBLP:journals/dm/ErwinH06, DBLP:journals/combinatorics/Boutin06, DBLP:journals/gc/FijavzM10}.
	
        A {\it distinguishing labeling} $\phi$ of $G$ colors the vertices of $G$ so that the only automorphism of the labeled graph $(G, \phi)$ is the identity map.  The {\it distinguishing number} of $G$, $D(G)$, is the fewest number of colors needed to create a distinguishing labeling of $G$.   For example,  $D(K_n) = n$, $D(P_n) = 2$ for $n \ge 2$ and $D(K_{n,n}) = n+1$. The notion of distinguishing labelings appeared as early as 1977 when Babai used them to create asymmetric infinite trees \cite{Ba77}, calling them {\it asymmetric colorings}.  Independently, Albertson and Collins \cite{AlCo96} defined distinguishing labelings and numbers for all graphs in 1996.  Since then hundreds of papers in graph theory and group theory have been published on this topic (e.g.,  \cite{KlWoZh06}, \cite{ ImSmTuWa15}, \cite{CoTr22},  \cite{ShAhTaHa22} and references therein).  

       A subset $S$ of vertices is a {\it fixing set} of $G$ if the only automorphism of $G$ that fixes every element in $S$ is the identity map. The {\it fixing number} of $G$, $Fix(G)$, is the size of a smallest fixing set.    A fixing set $S$ of $G$ can  be translated into  a distinguishing labeling $\phi_S$ by {\it individualizing} the vertices in $S$, assigning them distinct colors and assigning another color (e.g., the ``null" color) to the vertices not in $S$.   Thus,  $D(G)  \leq Fix(G) + 1$.   Sometimes, the equality holds but $D(G)$ and $Fix(G)$ can also be far apart.   For example, $Fix(K_n) = n-1$,  $Fix(P_n) = 1$ for $n \ge 2$ while $Fix(K_{n, n} )= 2(n-1)$ for $n \ge 2$.  Erwin and Harary \cite{DBLP:journals/dm/ErwinH06} defined fixing sets and numbers in 2006.  Around the same time, Boutin \cite{DBLP:journals/combinatorics/Boutin06} and Fijavz and Mohar \cite{DBLP:journals/gc/FijavzM10} wrote papers on the same subject, calling them {\it determining sets and numbers} and {\it rigid indices} respectively.  Like distinguishing labelings, many papers have been written about fixing sets. (e.g., \cite{DBLP:journals/jgt/Boutin09, DBLP:journals/combinatorics/CaceresGPS10, DBLP:journals/amc/GarijoHM14, GonzalezPuertas2019RemovingTwins} and the survey \cite{BaileyCameron2011BaseSizeMetricDimension}). 
       

       We are primarily interested in the algorithmic aspects of computing distinguishing and fixing numbers.  Let {\sc Dist}($G,k$) denote the problem: ``{\it Given a graph $G$ and an integer $k$, is $D(G) \leq k$?}''.  Define {\sc Fix}($G,k$)  similarly for the fixing numbers.  While these problems are not known to be in $P$ nor to be NP-complete, prior work  on the distinguishing numbers of  trees and forests \cite{ArDe04, Ch06}, planar graphs \cite{ArDe04, ArChDe08},  interval graphs \cite{Ch09} and unigraphs \cite{Ch25} suggests that when a graph class has an efficient isomorphism algorithm,  {\sc Dist}($G,k$) is solvable in polynomial-time.  One reason is that many isomorphism algorithms end up identifying the parts of a graph that are identical, which are also the parts that a distinguishing labeling must differentiate. Indeed, our strategy in \cite{Ch06, ArChDe08, Ch09} involved repurposing the isomorphism algorithms to count the number of inequivalent distinguishing labelings of a graph.  Another reason is that even when a graph class's isomorphism algorithm is not particularly useful, the fact that it is efficient suggests that the graph has a nice structure which can  be exploited to understand its  automorphisms and, consequently, compute its distinguishing number.  This was the case for unigraphs; Tyshkevich's canonical decomposition theorem \cite{Ty00} was the basis for our linear-time algorithm for computing the distinguishing number of a unigraph \cite{Ch25}.  Our current work falls in this category as well.

         Color refinement is a widely-used efficient heuristic for graph isomorphism.  As Grohe et al.~\cite{GroheKerstingMladenovSchweitzer2021ColorRefinement} state, ``more advanced graph isomorphism tests and almost all practical isomorphism tools use color refinement as a subroutine."   To determine if $G$ and $H$ are isomorphic, the heuristic starts with $\{V(G) \cup V(H) \}$ as the initial partition.  It then refines the  partition repeatedly based on the local neighborhoods of the vertices in the prior iteration.  Eventually, no more refinements can be made and the procedure ends.   The heuristic concludes that $G$ and $H$ are isomorphic if and only if $G$ and $H$ have the same number of vertices in each cell of the final partition. Unfortunately, color refinement can mistakenly conclude that two  graphs are isomorphic when they are not  (e.g. $C_6$ and $2C_3$).   
         
         In this paper, we focus on  {\it amenable graphs}.  A graph  $G$ is {\it amenable}  if, for any graph $H$, color refinement  correctly determines whether $G$ and $H$ are isomorphic or not.   That is, color refinement is the isomorphism algorithm for amenable graphs.  Arvind et al.~\cite{Arvind2017GraphIsomorphism} presented the first combinatorial characterization of amenable graphs and showed that they can be recognized in $O((n+m) \log n)$ time where $n = |V(G)|$ and $m = |E(G)|$.  The characterization involves meta-structures called {\it anisotropic components} which, remarkably, form rooted trees.   But the graphs induced by the anisotropic components can be complicated.  
        For our first contribution, we show that a series of transformations can be applied to the induced subgraphs so that an amenable graph $G$ can be viewed as the union of  {\it celled jellyfish graphs}.   Since jellyfish graphs are tree-like,  we generalize the approach in \cite{Ch06} to  compute the distinguishing number of each celled jellyfish graph.  We then combine the results to determine $D(G)$.  The algorithm runs in  $O((n+m) \log n)$ time.


       We also want to compute the fixing numbers of amenable graphs.  Since distinguishing labelings and fixing sets are closely related,  one would expect more interaction between the two topics. Unfortunately, we know of  only a few papers where this is the case \cite{DBLP:journals/combinatorics/AlbertsonB07, DBLP:journals/combinatorics/CaceresGPS10}.   For our second contribution, we show that the methodology described above is just as effective, and  less cumbersome, for computing the fixing numbers of amenable graphs which can also be done in  $O((n+m) \log n)$ time.  In light of this result, we suspect that the fixing numbers of planar graphs, interval graphs and unigraphs can also be computed in polynomial time.  
       
        Section 2 provides background information and preliminary results about jellyfish graphs.  Sections 3 and 4  tackle the problem of computing the distinguishing numbers of amenable graphs while Section 5 deals with the fixing number of amenable graphs.  
       

{\it Related works.}	A graph $G$ is {\it rigid} when the only automorphism of $G$ is the identity map.  Notice that $D(G) = 1$ and $F(G) = 0$ if and only if $G$ is rigid.  Thus, {\sc Dist}($G,1)$ and {\sc Fix}($G,0$) is at least as hard as {\sc Rigid}, the problem of determining if a graph is rigid. All three problems are in co-NP since a non-trivial automorphism of $G$ can be used to verify that $D(G) > 1, F(G) > 0$ and $G$ is not rigid.  In \cite{RuSu98},  Russell and Sundaram proved further that {\sc Rigid} and {\sc Dist}($G,k$) are in the the class AM, the set of languages for which there are Arthur and Merlin games.  They also noted that it is unlikely that   {\sc Dist}($G,k$) is co-NP hard.   No comparable results are known for {\sc Fix}($G,k$). 
         In \cite{DBLP:conf/mfcs/ArvindFKKR16}, Arvind et al.~also considered {\sc Fix}($G,k)$ from the perspective of parameterized complexity.  They showed that {\sc Fix}($G,k)$  is MINI[1]-hard when parameterized by $k$. We are not aware of any FPT results for distinguishing numbers.

 \section{Background and Preliminary Results}
 \label{sec:prelims}
 
 We begin by providing background information on color refinement and amenable graphs, distinguishing labelings and numbers, and fixing sets and numbers.  We also present new results about the distinguishing and fixing numbers of jellyfish graphs which will play an important role in the latter sections.

 \subsection{Color Refinement and Amenable Graphs}  

Let $G$ be a graph.  A {\it partition} of $V(G)$ is a collection of its subsets $\mathcal{P} = \{V_1, V_2, \hdots, V_k\}$ such that $\bigcup_{1 \leq i \leq k} V_i = V(G)$ and $V_i \cap V_j  = \emptyset$ for $i \neq j$.  Each element of the partition is called a {\it cell}.   A partition is {\it discrete} when each one of its cells has size one.  Given partitions $\mathcal{P}_1 = \{V_1, V_2, \hdots, V_k\}$ and $\mathcal{P}_2 = \{W_1, W_2, \hdots, W_\ell\}$,  we say $\mathcal{P}_1$ is {\it coarser} than $\mathcal{P}_2$ 
if for each $W_j$ there is a $V_i$ so that $W_j \subseteq V_i$ and $k < \ell$. 


For $X \subseteq V(G)$, let $G[X]$ denote the subgraph of $G$ induced by the vertices in $X$.  Similarly, for two disjoint subsets $X, Y$ of $V(G)$, let $G[X,Y]$ denote the bipartite subgraph induced by the vertices in $X$ and $Y$.   A partition $\mathcal{P} = \{V_1,  V_2,  \hdots, V_k\}$ of $V(G)$ is an {\it equitable partition} of $G$ if,  for all $V_i$ and for all $u, v \in V_i$, the vertices $u$ and $v$ have the same number of neighbors in {\it every} cell $V_j$,  including $V_i$ itself.  That is, $G[V_i]$  is a regular graph for each $i$ and   $G[V_i, V_j]$, is a biregular graph\footnote{All the vertices in $V_i$ have the same number of neighbors in $V_j$ and all the vertices in $V_j$ have the same number of neighbors in $V_i$.} for each $i, j$. Notice that  the discrete partition of $V(G)$ is always an equitable partition of $G$ while $\{V(G)\}$ is an equitable partition of $G$ if and only if $G$ is a regular graph.  All the equitable partitions of $G$ can be ordered using the {\it coarser than} relation and form a lattice \cite{McKayThesis, DBLP:journals/dm/RamanaSU94}.  Thus, $G$ has a {\it coarsest equitable partition} and a {\it finest equitable partition}, the latter of which is the discrete partition. 
   
    Two graphs $G$ and $H$ are {\it isomorphic}, $G \cong H$,  if they are structurally identical.  That is, there is a bijection $\pi: V(G) \rightarrow V(H)$ that preserves adjacencies --  for any two vertices $u, v \in V(G)$,  $uv \in E(G)$ if and only if $\pi(u) \pi(v) \in E(H)$.    The bijection $\pi$ is called an {\it isomorphism} from $G$ to $H$. When $H = G$, the bijection $\pi$ is referred to as an {\it automorphism} of $G$.  The automorphism is {\it trivial} if it is the identity map $Id$; otherwise, it is {\it non-trivial}.  The set of automorphisms of $G$, $Aut(G)$,  form a group under the composition of functions so it is called the {\it automorphism group} of $G$.    When
  $G$ and $H$ are colored, their isomorphisms and automorphisms also preserve colors. 

    {\it Color refinement} or CR  (also known as the {\it 1-dimensional Weisfeiler-Leman algorithm}) is a popular heuristic for graph isomorphism \cite{GroheKerstingMladenovSchweitzer2021ColorRefinement}.  But fundamentally it is an algorithm that finds an equitable partition of a graph.  Given graph $G$, it starts with an initial partition $\mathcal{P}_0$ of $V(G)$. It then refines the partition iteratively until an equitable partition of $G$ is obtained.  
      Specifically, at iteration $i$, the partition $\mathcal{P}_i$ is constructed from the previous partition $\mathcal{P}_{i-1}$ as follows: two vertices are part of the same cell in $\mathcal{P}_i$ if and only if they are from the same cell of $\mathcal{P}_{i-1}$ and  have the same number of neighbors in each cell of $\mathcal{P}_{i-1}$.  The algorithm ends when $\mathcal{P}_i = \mathcal{P}_{i-1}$.    When $\mathcal{P}_0 = \{V(G)\}$,  the final partition $\mathcal{P}_G$ is referred to as the {\it stable partition} of $G$.   It is clearly an equitable partition.  In fact, it is the {\it coarsest} equitable partition of $G$  \cite{McKayThesis}.  Figure \ref{ExampleFig} shows  an example of how color refinement works.  Color refinement can be implemented in $O((n+m) \log n)$ time where $n = |V(G)|$ and $m = |E(G)|$ \cite{CardonCrochemore1982PartitioningAGraph}.   
    
    To determine if two graphs $G$ and $H$ are isomorphic,  color refinement is applied to $G \cup H$.   If $G \cong H$,  then $|V(G)\cap X| =  |V(H)\cap X|$ for every cell $X$ of $\mathcal{P}_{G \cup H}$.   Thus,  if  there is some cell $X$ so that $|V(G)\cap X| \neq  |V(H)\cap X|$,  the algorithm concludes that $G \not \cong H$.  Otherwise, it concludes that $G \cong H$.   Unfortunately, the latter part can sometimes be wrong.  It is known that CR cannot distinguish two graphs that are {\it fractionally isomorphic} \cite{DBLP:journals/dm/RamanaSU94} such as $C_6$ and $2C_3$.   For this reason, CR is often just referred to as a heuristic for graph isomorphism.

 
Nonetheless,  as Arvind et al.~noted in \cite{Arvind2017GraphIsomorphism}, CR works correctly for some graph classes like unigraphs and trees.  They then investigated the graphs  for which CR is a perfect graph isomorphism algorithm.   A graph  $G$ is {\it amenable}  if, for any graph $H$, CR correctly determines whether $G$ and $H$ are isomorphic or not.   We now present their combinatorial characterization of amenable graphs.


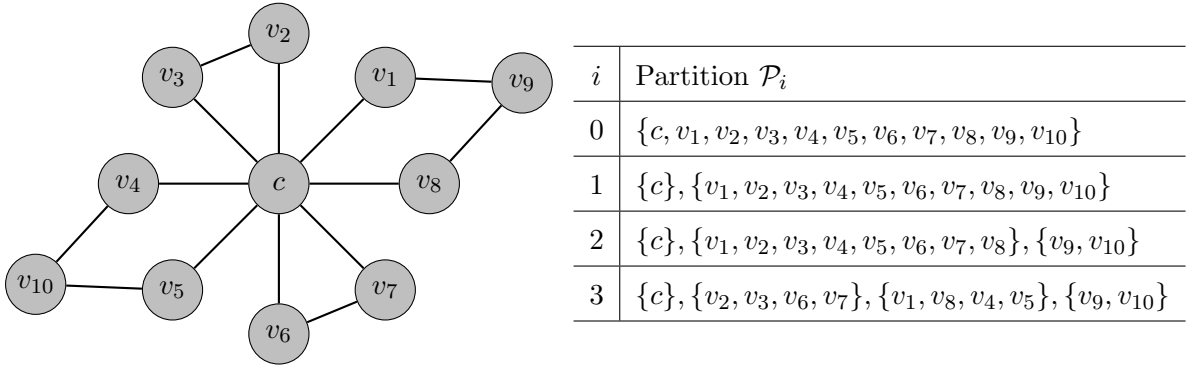
\begin{figure}
\begin{minipage}{0.45\textwidth} 
\begin{tikzpicture}[
    vertex/.style={circle, draw, fill=lightgray, minimum size=0.8cm, inner sep=0pt},
    edge_style/.style={thick}
]

  \node[vertex, fill=lightgray] (center) at (0,0) {$c$}; 

  \def\radius{2cm} 

  \foreach \i in {1,...,8} {
    \node[vertex] (v\i) at (\i*45:\radius) {$v_\i$}; 
  }

\node[vertex] (v9) at (22.5:3.5cm) {$v_9$};
\node[vertex] (v10) at (202.5:3.5cm) {$v_{10}$};

  \foreach \i in {1,...,8} {
    \draw[edge_style] (center) -- (v\i);
  }
  
  \draw[edge_style] (v1) -- (v9);
  \draw[edge_style] (v8) -- (v9);
  \draw[edge_style] (v4) -- (v10);
  \draw[edge_style] (v5) -- (v10); 
   \draw[edge_style] (v2) -- (v3);
    \draw[edge_style] (v6) -- (v7); 
\end{tikzpicture}
 \end{minipage}
 \begin{minipage}{0.45\textwidth} 
  \centering
  \begin{tabular}{c|l}
  \hline
   $i$  & Partition $\mathcal{P}_i$ \\
  \hline
   0 & $\{c, v_1, v_2, v_3, v_4, v_5, v_6, v_7, v_8, v_9, v_{10}\}$ \\
  \hline
   1 & $\{c\}, \{v_1, v_2, v_3, v_4, v_5, v_6, v_7, v_8, v_9, v_{10}\}$ \\
  \hline
  2 & $\{c\}, \{v_1, v_2, v_3, v_4, v_5, v_6, v_7, v_8 \}, \{v_9, v_{10} \}$ \\
  \hline
   3 & $\{c\}, \{v_2, v_3, v_6, v_7\}, \{v_1, v_8, v_4, v_5\}, \{v_9, v_{10} \}$ \\
  \hline
\end{tabular}
\end{minipage}
\caption{Color refinement is applied to graph $G$ on the left with $\mathcal{P}_0 = \{V(G)\}$. }
\label{ExampleFig}
\end{figure}


 
 
 


A graph or a bipartite graph is {\it empty} if it has no edges.  A {\it matching graph} is a graph of the form $rK_2$ for some $r \ge 2$.  
 We start with the local structure of amenable graphs.

\begin{lemma} (\cite{Arvind2017GraphIsomorphism}, Lemma 3)
Let $G$ be an amenable graph.  The stable partition $\mathcal{P}_G$ has the following properties:

\noindent (A) For any cell $X \in \mathcal{P}_G$,  $G[X]$ is a $5$-cycle,  an empty graph, a complete graph,   a matching graph or its complement; 

\noindent (B) For any two cells $X, Y \in \mathcal{P}_G$, $G[X,Y]$ is an empty graph, a complete bipartite graph, a disjoint union of stars $sK_{1,t}$ where the $s$ centers of the stars  are on one side and  the $st$ leaves of the stars are on the other side,  or the bipartite complement of the last graph.  
\label{local}
\end{lemma} 
 
 Arvind et al.~obtained the above lemma by noting that the $G[X]$'s and the $G[X,Y]$'s have to be regular unigraphs and biregular unigraphs respectively. Otherwise, they can be replaced by other regular or biregular graphs to create a new graph $H$,  and CR will not be able to distinguish between $G$ and $H$.

  To understand the global structure of an amenable graph $G$, we consider the {\it cell graph} $C(G)$ of $G$.  It is a complete graph whose vertex set is $\mathcal{P}_G$; i.e., the vertices are the cells of $\mathcal{P}_G$.   A vertex $X$ of $C(G)$ is {\it homogeneous} if $G[X]$ is an empty graph or a complete graph; otherwise, it is {\it heterogeneous}.  An edge $XY$ of $C(G)$ is {\it isotropic} if $G[X,Y]$ is an empty graph or a complete graph; otherwise, it is {\it anisotropic}.  
  
  It turns out that what matters in $C(G)$ are the anisotropic edges.  Starting with $C(G)$, delete all the isotropic edges.  The resulting connected components are called the  {\it anisotropic components} of $C(G)$.  
  The next lemma notes that each anisotropic component is a tree with some nice properties.


\begin{lemma} (\cite{Arvind2017GraphIsomorphism}, Lemma 8)
Let $G$ be an amenable graph.  Then for any anisotropic component $A$ of $C(G)$, the following is true.

\noindent (C) $A$ is a tree.  Furthermore, if $A$ is rooted at a cell whose cardinality is the least, then for any edge $XY$ where $X$ is the parent of $Y$,  $|X| \leq |Y|$.  That is, when a path is traversed from the root down the tree, the sizes of the cells are monotonically non-decreasing. 

\noindent (D) $A$ contains at most one heterogeneous vertex, and this vertex has the least cardinality among the cells in $A$.  
\label{global}
\end{lemma}

Here now is Arvind et al.'s  characterization of amenable graphs.

\begin{theorem} (\cite{Arvind2017GraphIsomorphism}, Theorem 9)
A graph $G$ is amenable if and only if $G$ satisfies conditions (A), (B), (C) and (D).  Consequently, amenable graphs can be recognized in $O((n+m)\log n)$ time where $n = |V(G)|$ and $m = |E(G)|$. 
\label{mainchar}
\end{theorem}

Consider graph $G$ in Figure \ref{ExampleFig}.  We show its cell graph $C(G)$ in Figure \ref{ExampleFig2}.  The induced graphs $G[X_i]$ for $i = 1, 3, 4$ are empty graphs while $G[X_2]$ is $2K_2$ and the only heterogenous cell.   As for $G[X_i, X_j]$'s, with the exception of $G[X_3, X_4]$, all the induced bipartite graphs are empty or complete bipartite graphs.  On the other hand, $G[X_3, X_4]$ is $2K_{1,2}$ and is the only anisotropic edge.  Thus, $G$ satisfies conditions (A) and (B).  

Once the isotropic edges in $C(G)$ are removed, there are three components left:  $X_1$ by itself, $X_2$ by itself and $X_3$ and $X_4$ together.    Since the first two components contain only one cell, they satisfy conditions (C) and (D) trivially.  For the third component, $|X_4| < |X_3|$ so the component should be rooted at $X_4$.  There is only one path in the component and the cell sizes are monotonically non-decreasing. Furthermore, the component has no heterogeneous vertices.  So the third component also satisfies conditions (C) and (D).   From Theorem \ref{mainchar}, $G$ is an amenable graph.

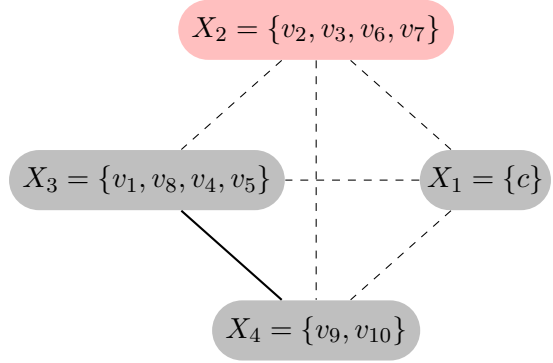
\begin{figure}

\centering
\begin{tikzpicture}[
  every node/.style={draw=none, rounded rectangle, minimum width=2cm, minimum height=0.8cm},
  level distance=1.5cm,
  sibling distance=2cm
  ]

\node[fill=lightgray] (X1) at (4.5,0) {$X_1 = \{c\}$};
\node[fill=pink] (X2) at (2.25,2) {$X_2 = \{v_2, v_3, v_6, v_7\}$};
\node[fill=lightgray] (X3) at (0,0) {$X_3 = \{v_1, v_8, v_4, v_5\} $};
\node[fill=lightgray] (X4) at (2.25,-2) {$X_4 = \{v_9, v_{10}\}$};

 \draw[dashed] (X1) -- (X2);
\draw[dashed] (X1) -- (X3);
\draw[dashed] (X1) -- (X4);
\draw[dashed] (X2) -- (X3);
\draw[dashed] (X2) -- (X4);
\draw[thick] (X3) -- (X4);
 \end{tikzpicture}
 \caption{The cell graph $C(G)$ of $G$ in Figure \ref{ExampleFig}.  The cells in gray and pink are homogeneous and heterogeneous cells respectively.  The dashed and thick lines are isotropic and anisotropic edges respectively. Thus, $C(G)$ has three anisotropic components: $X_1$ by itself, $X_2$ by itself and $X_3$ and $X_4$ together.}
 \label{ExampleFig2}
\end{figure}

\subsection{Distinguishing Labelings}
\label{sec:prelim}



Let $\phi:V(G) \rightarrow \{1, 2, \hdots, c\}$ be a  vertex labeling of $G$ that uses $c$ colors and denote the labeled graph as $(G, \phi)$.  We say that $\phi$ is a {\it distinguishing labeling} if, for every nontrivial automorphism $\pi$ of $G$,  there is some vertex $v$ so that $\phi(v) \neq \phi(\pi(v))$.  In other words, $\phi$ has {\it broken} all the non-trivial automorphisms of $G$ so that $Aut(G, \phi) = \{Id\}$. 
The {\it distinguishing number of $G$}, $D(G)$, is the fewest number of colors needed to create a distinguishing labeling of $G$.


One surprisingly useful method for computing distinguishing numbers is counting.  Let $\phi$ and $\phi'$ be two distinguishing labelings of $G$.  They are {\it equivalent} when there is some $\pi \in Aut(G)$ so that  $\phi(v) = \phi'(\pi(v))$ for each vertex $v$ of $G$; otherwise, they are {\it inequivalent}.  That is,  $\phi$ and $\phi'$ are equivalent if and only if $(G, \phi)$ can be made to look like $(G, \phi')$.  Given $c$ colors, let $D(G, c)$ denote the number of pairwise inequivalent distinguishing labelings of $G$ that use at most  $c$ colors.


\begin{proposition}
Let $F = rG$.   Then $D(F) = \min \{c: D(G,c) \ge r\}$.  
\label{propcount}
\end{proposition}


Consider $F = rK_2$.   Two colors are needed to create a distinguishing labeling $\phi$ for $K_2$.  Another distinguishing labeling $\phi'$ of $K_2$ is inequivalent to $\phi$ if and only if they use different pairs of colors. Thus, $D(K_2,c) = \binom{c}{2}$ and $D(F) = \min \{c: \binom{c}{2} \ge r\}$. 




But there is still the question of how $D(G,c)$ can be computed in general.  When $T$ is a rooted tree, $D(T,c)$ can be computed recursively as shown in the next theorem.   Any automorphism of $T$ maps the root $r(T)$ to itself so what matters are the subtrees   rooted at the children of $r(T)$.


\begin{theorem} \cite{ArDe04, Ch06}
Let $T$ be a rooted tree and $r(T)$ be its root.  Let $\mathcal{T}$ contain the subtrees rooted at the children of $r(T)$.  Suppose $\mathcal{T}$ has exactly $g$ distinct isomorphism classes and the $j$th isomorphism class has $m_j$ copies of the rooted tree $T_j$; 
i.e., $\mathcal{T} = m_1 T_{1} \cup m_2 T_{2} \cup \cdots \cup m_g T_{g}.$  Then 
$$D(T,c) = c \prod_{j = 1}^g \binom{D(T_{j},c)}{m_j}$$
and 
$$D(T) = \min\{c: D(T,c) > 0\}.$$  

\label{disttree}
\end{theorem}

Cheng \cite{Ch06} and Arvind and Devanur \cite{ArDe04} showed that computing $D(T,c)$ for a fixed $c$ can be done in $O(n)$ time.  By doing binary search over the range $[1, n]$,  $D(T)$ can be computed in $O(n \log n)$ time.  Seara et al.~\cite{distinguishing2012} improved the running time further to $O(n)$.

Later on,  we will encounter graphs that we informally call  {\it jellyfish graphs}.  Every jellyfish graph $J$ can be described using two graphs $H$ and $L$. 
  The {\it head of $J$} is $H$, and $H$ contains a Hamiltonian cycle.  Rooted at each vertex  of $H$ is a {\it leg} that is isomorphic to $L$,  a rooted tree. We shall refer to the particular leg rooted at vertex $v$  as $L_v$.  No two legs share a vertex.   A consequence of a jellyfish graph's  structure is that every automorphism of $J$ maps the head to itself and a leg to itself or another leg of $J$.   See Figure \ref{fig:jellyfish} for an example.

\begin{figure}
\centering
\begin{tikzpicture}[
  outer/.style={draw=none, fill=pink, rounded rectangle, minimum width=2.4cm, minimum height=0.65cm},
  inner/.style={circle, fill=black, inner sep=3pt},
  every path/.style={thick}
]

\node[outer, minimum width=11cm, minimum height=1.5cm] (root) at (0,0) {};

\node[inner] (r1) at ($(root)+(0,-0.3)$) {};
\node[inner] (r2) at ($(root)+(2,-0.3)$) {};
\node[inner] (r3) at ($(root)+(4,-0.3)$) {};
\node[inner] (r4) at ($(root)+(-2,-0.3)$) {};
\node[inner] (r5) at ($(root)+(-4,-0.3)$) {};

\foreach \i in {1, 2, 3, 4, 5}{
	\node[inner] (c\i1) at ($(r\i)+(0.5,-1)$) {};
	\node[inner] (c\i2) at ($(r\i)+(-0.5,-1)$) {};
	\node[inner] (c\i3) at ($(c\i2)+(-0.5,-1)$) {};
	\node[inner] (c\i4) at ($(c\i2)+(0.5,-1)$) {};
	\draw (r\i) -- (c\i1);
	\draw (r\i) -- (c\i2);
	\draw (c\i2) -- (c\i3);
	\draw (c\i2) -- (c\i4);
}

 \draw (r1) -- (r2);
\draw (r2) -- (r3);
\draw (r5) -- (r4);
\draw (r1) -- (r4);
\draw (r5) to [bend left =20] (r3);
\end{tikzpicture}
\caption{ An example of a jellyfish graph.  The shaded portion is the head.  Rooted at each vertex of the head is a leg.  The legs are pairwise isomorphic and disjoint.}
\label{fig:jellyfish}
\end{figure}
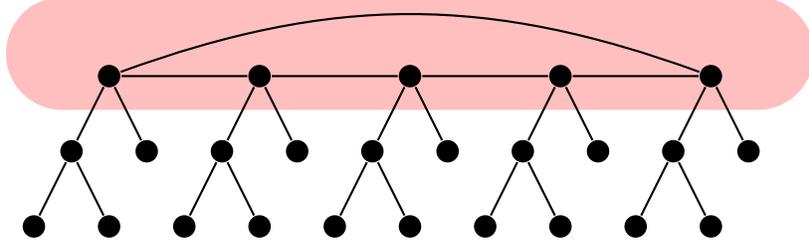

Let $\phi$ be a labeling of the jellyfish graph $J$.   Notice that we can specify $\phi$ by  the labelings it assigns to the legs of $J$.  For $v \in V(H)$, let $\phi_v$ be the labeling of $L_v$ when $\phi$ is restricted to its vertices. 
The leg labelings in turn induce a labeling $\phi_{proj}$ on $H$ as follows:  for any $v, w \in V(H)$, let $\phi_{proj}(v) = \phi_{proj}(w)$ if and only if the labeled leg $(L_v, \phi_v)$ is isomorphic to the labeled leg $(L_w, \phi_w)$ or, equivalently, there is an isomorphism from $L_v$ to $L_w$ that preserves the labels of $\phi$.  We call $\phi_{proj}$  the {\it projection of $\phi$ onto $H$}.\footnote{Technically, there can be many projections of $\phi$ onto $H$ because we do not specify the colors used by $\phi_{proj}$.}
To determine if $\phi$ is a distinguishing labeling of $J$,  we have to examine its behavior on each leg $L_v$ and the head $H$.   

\begin{lemma}
Let $J$ be a jellyfish graph with head $H$ and legs isomorphic to the tree $L$.   A labeling $\phi$ of $J$ is distinguishing if and only if (i) $\phi_v$ is a distinguishing labeling of $L_v$ for each $v \in V(H)$ and (ii) $\phi_{proj}$ is a distinguishing labeling of $H$. 
\label{lemmajellyfish}
\end{lemma}

\begin{proof}
Let  $\phi$ be a labeling of $J$.  
If $\phi_v$ is not a distinguishing labeling of $L_v$, then  $(L_v, \phi_v)$ has a non-trivial automorphism $\pi_v$.  This $\pi_v$ can be extended to a non-trivial automorphism $\pi$ of $(J, \phi)$ as follows: let $\pi(a) = \pi_v(a)$ for every vertex $a$ in $L_v$ and let $\pi(b) = b$ for every vertex $b$ not in $L_v$.   
Hence,  $\phi$ is not a distinguishing labeling of $J$.  


Suppose $\phi_{proj}$ is not a distinguishing labeling of $H$. Then $(H, \phi_{proj})$ has a non-trivial automorphism $\pi_H$.  This means that for every vertex $v$ of $H$, $\phi_{proj}(v) = \phi_{proj}(\pi_H(v))$.  Consequently,    there is a mapping $\tau_v$ from $L_v$ to $L_{\pi_H(v)}$ that preserves adjacencies and the labels of $\phi$.  When $\pi_H(v) \neq v$, the mapping $\tau_v$ is an isomorphism; otherwise, $\tau_v$ is an automorphism. Regardless, the set $\{\tau_v, v \in V(H)\}$ can be used to define an automorphism $\pi$ on $J$ that preserves the labels of $\phi$.  Since  $\pi_H$ is non-trivial,  $\pi$ is also non-trivial.  As a result, $\phi$ is not a distinguishing labeling of $J$.  We have now shown that when $\phi$ is a distinguishing labeling of $J$, conditions (i) and (ii) must hold.  

Consider the converse.  Assume $\phi$ is a labeling of $J$ that satisfies conditions (i) and (ii).  Let $\pi$ be an automorphism of the labeled graph $(J, \phi)$.  Clearly, $\pi$ maps $H$ to itself.  Additionally,  for each $v \in V(H)$, $(L_v, \phi_v)$ is isomorphic to $(L_{\pi(v)}, \phi_{\pi(v)})$. This implies that for each $v \in V(H)$, $\phi_{proj}(v) = \phi_{proj}(\pi(v))$. That is, $\pi$ when restricted to $H$ is an automorphism of $(H, \phi_{proj})$.  From condition (ii), this automorphism is trivial so $\pi(v) = v$ for each $v \in V(H)$.   

Now, if $\pi(v) = v$ for each $v \in V(H)$, then $\pi$ must map each leg $L_v$  to itself.  Another way of putting it is $\pi$ when restricted to  $L_v$ is an automorphism $\pi_v$ of $(L_v, \phi_v)$.  From condition (i), $\pi_v$ is  trivial. Since this property holds for each $v \in V(H)$, $\pi$ itself must be a trivial automorphism of $(J, \phi)$.   Thus, we have shown that when $\phi$ satisfies (i) and (ii),  $\phi$ is a distinguishing labeling of $J$.
\end{proof}

\begin{figure}
\centering
\begin{tikzpicture}[
  outer/.style={draw=none, fill=pink, rounded rectangle, minimum width=2.4cm, minimum height=0.65cm},
  inner/.style={circle, fill=black, inner sep=3pt},
  every path/.style={thick}
]

\node[outer, minimum width=11cm, minimum height=1.5cm] (root) at (0,0) {};

\node[inner] (r1) at ($(root)+(0,-0.3)$) {};
\node[inner] (r2) at ($(root)+(2,-0.3)$) {};
\node[inner] (r3) at ($(root)+(4,-0.3)$) {};
\node[inner] (r4) at ($(root)+(-2,-0.3)$) {};
\node[inner] (r5) at ($(root)+(-4,-0.3)$) {};



\foreach \i in {1, 2, 3, 4, 5}{
	\node[inner] (c\i1) at ($(r\i)+(0.5,-1)$) {};
	\node[inner] (c\i2) at ($(r\i)+(-0.5,-1)$) {};
	\node[inner] (c\i3) at ($(c\i2)+(-0.5,-1)$) {};
	\node[inner, fill=lightgray] (c\i4) at ($(c\i2)+(0.5,-1)$) {};
	\draw (r\i) -- (c\i1);
	\draw (r\i) -- (c\i2);
	\draw (c\i2) -- (c\i3);
	\draw (c\i2) -- (c\i4);
}

\node[inner, fill=lightgray] (c42) at ($(r4)+(-0.5,-1)$) {};
\node[inner, fill=lightgray] (c22) at ($(r2)+(-0.5,-1)$) {};
\node[inner, fill=lightgray] (c31) at ($(r3)+(+0.5,-1)$) {};

 \draw (r1) -- (r2);
\draw (r2) -- (r3);
\draw (r5) -- (r4);
\draw (r1) -- (r4);
\draw (r5) to [bend left =20] (r3);
\end{tikzpicture}
\caption{ A distinguishing labeling of the jellyfish graph in Figure \ref{fig:jellyfish} that uses two colors (gray and black).  Three inequivalent labelings are used on the five legs.  They in turn induce a labeling on the head, a $5$-cycle, that is distinguishing.}
\label{fig:jellyfish2}
\end{figure}
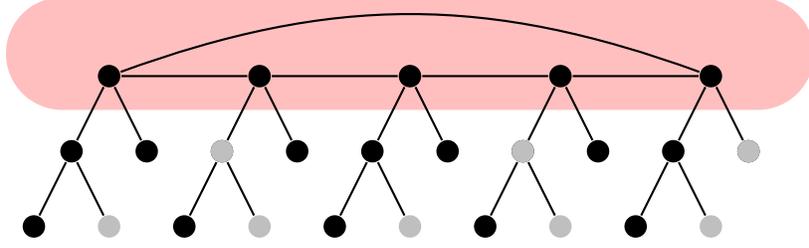

\begin{theorem}
\label{distjellyfish}
Let $J$ be a jellyfish graph with head $H$ and legs isomorphic to the tree $L$.   Let $D(H) = d^*$.  Then $D(J) = \min\{c: D(L, c) \ge d^* \}. $
\end{theorem}

\begin{proof}
We create a labeling $\phi$ of $J$ as follows:  Let $\phi'$ be a distinguishing labeling of $H$ that uses colors from the set $\{1, 2, \hdots, d^*\}$. 
  Denote as $V_i$ the set of vertices $v$ in $H$  so that $\phi'(v) = i$.  Find $d^*$ pairwise inequivalent distinguishing labelings of $L$ and denote them as $\psi_1, \psi_2, \hdots, \psi_{d^*}$.  Then for each $V_i$ and for each $v \in V_i$, label the leg $L_v$ with $\psi_i$.  By  Lemma \ref{lemmajellyfish}, $\phi$ is a distinguishing labeling of $J$.  The minimum number of colors needed to create $\phi$ is $\min\{c: D(L, c) \ge d^* \}$ so $D(J) \leq \min\{c: D(L, c) \ge d^* \}. $

On the other hand, according to Lemma \ref{lemmajellyfish},  every distinguishing $\phi$ of $J$ has to have a projection $\phi_{proj}$ that uses at least $d^*$ colors.  This means that the collection $\{\phi_v, v \in V(H)\}$ contains at least $d^*$ pairwise inequivalent distinguishing labelings of $L$. So $\phi$ has to use at least $\min\{c: D(L, c) \ge d^* \}$ colors. Hence, $D(J) \ge \min\{c: D(L, c) \ge d^* \}. $
\end{proof}

Figure \ref{fig:jellyfish2} shows an example of a $2$-distinguishing labeling of the jellyfish graph in Figure \ref{fig:jellyfish}.  The graph is clearly not rigid so its distinguishing number is $2$.  

\subsection{Fixing Sets}

A subset $S$ of $V(G)$ is a {\it fixing set} of $G$ (or $S$ {\it fixes} $G$) if the only  automorphism $\pi \in Aut(G)$ such that $\pi(s) = s$ for every $s \in S$ is the identity map.
 The {\it fixing number of $G$}, $Fix(G)$, is the size of a smallest fixing set.  We note that $Fix(G) = 0$ when $G$ is a rigid graph and $Fix(G) > 0$ otherwise.   It is easy to check that the next proposition is true:

\begin{proposition}
Let $S \subseteq V(G)$.  Let $\phi_S$ be the vertex labeling that assigns distinct colors to the vertices in $S$ and another color (the ``null" color) to the all the vertices not in $S$.  Then $S$ is a fixing set of $G$ if and only if $\phi_S$ is a distinguishing labeling of $G$.   Thus, $D(G) \leq Fix(G) + 1$.  
\label{proptranslate}
\end{proposition}



Viewing a fixing set as a distinguishing labeling, we shall say that $S \subseteq V(G)$ {\it breaks} an automorphism $\pi$ of $G$ if for some $v \in S$, $\pi(v) \neq v$.  When $S$ breaks every non-trivial automorphism of $G$, then $S$ is a fixing set of $G$.  

In spite of Proposition \ref{proptranslate},  $Fix(G)$ and $D(G)$ can be quite different. The main reason is that distinguishing labelings can reuse colors whereas fixing sets, when viewed as distinguishing labelings, do not reuse colors with the exception of the null color.  Thus, for many graphs, $D(G)$ is much smaller than $Fix(G)$.  On the other hand,  the same reason makes the computation of $Fix(G)$ easier than $D(G)$ as illustrated by the results below.  In \cite{DBLP:journals/dm/ErwinH06},  Erwin and Harary made the following important observation about fixing numbers. 

\begin{proposition}
\cite{DBLP:journals/dm/ErwinH06}  Let $A$ and $B$ contain the connected components of $G$ that are rigid and non-rigid respectively.  Let $k$ be the number of isomorphism classes in $A$.  Then $$Fix(G) = \sum_{Y \in B} Fix(Y) + |A| - k.$$
\label{propErwinHarary}
\end{proposition}

Let $Y$ be a connected component of $G$.  The formula above stems from the fact that when $S$ is a fixing set of $G$, $S \cap V(Y)$ has to be a fixing set of $Y$ too.  When $Y$ is not rigid,  $|S \cap V(Y)| \ge 1$.  This is enough to guarantee that if there is another connected component $Y'$ isomorphic $Y$, 
$S$ also breaks automorphisms of $G$ that map $Y$ to $Y'$.  Thus, for the components in $B$, the number of vertices in $S$ is $ \sum_{Y \in B} Fix(Y)$.   However, when $Y$ is rigid,  it is possible for $|S \cap V(Y)|  = 0$.  If $Y'$ above exists and $|S \cap V(Y')|  = 0$ too, 
then $S$ has not broken all non-trivial automorphisms of $G$.  Thus, for the isomorphism class that contains $Y$,  all but one of the copies of $Y$ must have a vertex in $S$.   Since $A$ has $k$ isomorphism classes, the number of vertices in $S$ that are also in its components  is $|A| - k$.   We now apply Proposition \ref{propErwinHarary} to our setting.

\begin{proposition}
Let $F = rG$.  Then $Fix(F) = r-1$ if $G$ is a rigid graph and is equal to $r \times Fix(G)$ otherwise.  
\label{propcount2}
\end{proposition}



In the theorem below,  we need to differentiate the trees that are rigid from those that are not.  We use $\mathds{1}_E$ as the indicator variable for event $E$.


\begin{theorem} 
Let $T$ be a rooted tree and $r(T)$ be its root.  Let $\mathcal{T}$ contain the subtrees rooted at the children of $r(T)$.  Suppose $\mathcal{T}$ has exactly $g$ distinct isomorphism classes and the $j$th isomorphism class has $m_j$ copies of the rooted tree $T_j$;  i.e., $\mathcal{T} = m_1 T_{1} \cup m_2 T_{2} \cup \cdots \cup m_g T_{g}.$    Then
$$ Fix(T) = \sum_{j=1}^g  \left[(m_j -1) \cdot \mathds{1}_{(Fix(T_j) = 0)} + \left(m_j \times Fix(T_j)\right) \cdot  \mathds{1}_{(Fix(T_j) \neq 0)} \right]. $$

\label{fixtree}
\end{theorem}

Thus,  like $D(T)$,  $Fix(T)$ can be computed recursively but in a more direct way.  It does not involve counting inequivalent fixing sets.

We now extend Lemma \ref{lemmajellyfish} to fixing sets.  Let $J$ be a jellyfish graph and $S \subseteq V(J)$.  For $v \in V(H)$,  let $S_v$ contain all the vertices of $S$ that are part of the leg $L_v$ so $S =  \bigcup \{S_v, v \in V(H)\}$. The $S_v$'s then induce a  projection $S_{proj}$ on the head $H$  as follows:  $v \in S_{proj}$ if and only if $S_v \neq \emptyset$.  

\begin{lemma}
Let $J$ be a jellyfish graph with head $H$ and legs isomorphic to the tree $L$.  A subset $S$ of $V(J)$ is a fixing set of $J$ if and only if (i) $S_v$ is a fixing set of $L_v$ for each $v \in V(H)$ and (ii) $S_{proj}$ is a fixing set of $H$.  
\label{lemmajellyfish2}
\end{lemma}

 We omit the proof as it can be obtained by converting $S$ into $\phi_S$ and applying Lemma \ref{lemmajellyfish}.

\begin{theorem}
\label{thmjellyfish2}
Let $J$ be a jellyfish graph with head $H$ and legs isomorphic to $L$.  When $L$ is a rigid graph, $Fix(J) = Fix(H)$; otherwise, $Fix(J) = |V(H)| \times Fix(L)$. 
 Equivalently, 
$$ Fix(J) =  Fix(H) \cdot \mathds{1}_{(Fix(L) = 0)} + |V(H)| \times Fix(L) \cdot \mathds{1}_{(Fix(L) \neq 0)}. $$ 
\end{theorem}

\begin{proof}
Let us first find a fixing set for $J$.  When $L$ is a rigid graph,  let $S$ be a fixing set of $H$.  Then $S$ is automatically a fixing set of $J$ by Lemma \ref{lemmajellyfish2}.  In this case, $J$ has a fixing set of size $Fix(H)$. When $L$ is not rigid,  let  $S_v$ be a fixing set of $L_v$ for each $v \in V(H)$.   Notice that  each $S_v$ is non-empty so $S_{proj} = V(H)$ and therefore a fixing set of $H$.  By Lemma \ref{lemmajellyfish2}, $S = \bigcup_{v \in V(H)} S_v$ is a fixing set of $J$.  This implies that $J$ has a fixing set of size $|V(H)| \times Fix(L)$.  


This time, we bound $Fix(J)$ from below. Suppose $S$ is a minimum-sized fixing set of $J$.  Each $S_v$ is also a fixing set of $L_v$ so $|S_v| \ge Fix(L)$.  Hence, $Fix(J)  \ge |V(H)| \times Fix(L)$.  By Lemma \ref{lemmajellyfish2}, $S_{proj}$ is also a fixing set of $H$.  But $v \in S_{proj}$ if and only if some vertex of $L_v$ is part of $S$.  Thus, $Fix(J) \ge Fix(H)$.   When $L$ is rigid, the second lower bound is larger because $Fix(L) = 0$.  Combining with the result from the previous paragraph, $Fix(J) = Fix(H)$.  When $L$ is not rigid, the first lower bound is larger since $|V(H)| \ge Fix(H)$.  Thus, we have $Fix(J) =  |V(H)| \times Fix(L)$. 
\end{proof}

We leave it up to the reader to verify that the fixing number of the jellyfish graph in Figure \ref{fig:jellyfish} is $5$. 

 \section{The Distinguishing Numbers of Amenable Graphs, Part 1}

Shortly, we will express the distinguishing number of an amenable graph $G$ in terms of the distinguishing numbers of its subgraphs.  The stable partitions of the subgraphs, however, are not necessarily consistent with $\mathcal{P}_G$, the stable partition of $G$.  We need $\mathcal{P}_G$ to ``carry-over" as we deal with the subgraphs. 
 Thus, we introduce the notion of {\it celled graphs}.  
 
Let $G$ be a graph and  $\mathcal{P}$ be a partition of $V(G)$.  Then $(G, \mathcal{P})$ is a {\it celled graph}; that is, $G$ is a graph whose vertices are members of the cells in $\mathcal{P}$.   In many cases, we will also consider graphs $H$ with $V(H) \subseteq V(G)$.  Instead of using a partition of $V(H)$, we will just keep the partition of $V(G)$ and denote the celled version of $H$ as $(H, \mathcal{P})$.

  Let $Aut(G, \mathcal{P})$ contain all the automorphisms $\pi$ of $G$ that are {\it cell-preserving}; i.e., for any $v \in V(G)$, $v$ and $\pi(v)$ are in the same cell of $\mathcal{P}$.  Notice that $Id \in Aut(G, \mathcal{P})$ and  $Aut(G, \mathcal{P}) \subseteq Aut(G)$.   A labeling $\phi$ of $G$ is {\it $\mathcal{P}$-distinguishing} if $\phi$ breaks every non-trivial automorphism in $Aut(G, \mathcal{P})$.  Furthermore, $D(G, \mathcal{P})$ is the fewest number of colors needed for a $\mathcal{P}$-distinguishing labeling of $G$.  We shall also use $D((G, \mathcal{P}), c)$ for the number of pairwise inequivalent $\mathcal{P}$-distinguishing labelings of $G$ that use at most $c$ colors.

\begin{lemma}
( \cite{Arvind2017GraphIsomorphism}, Lemma 1)  Suppose $\pi$ is an isomorphism from $G$ to $H$.   When color refinement is applied to $G \cup H$,    $v$ and $\pi(v)$ belong to the same cell in every iteration of the algorithm for any $v \in V(G)$.
\label{lemmacell}
 \end{lemma}
 
 
 
 \begin{corollary}
 For any graph $G$,  $Aut(G) = Aut(G, \mathcal{P}_G)$. 
 \label{cor1}
 \end{corollary}
 
 \begin{proof}
 If we replace $H$ with $G$ in Lemma \ref{lemmacell}, we get that for any automorphism $\pi$ of $G$ and for any $v \in V(G)$, $v$ and $\pi(v)$ will be in the same cell of $\mathcal{P}_G$.  That is, every automorphism of $G$ is cell-preserving with respect to $\mathcal{P}_G$.  It follows that $Aut(G) \subseteq Aut(G, \mathcal{P}_G)$.  But $Aut(G, \mathcal{P}_G) \subseteq Aut(G)$.  Thus, the two sets are equal.  
 \end{proof}

In an earlier paper, we proved the following:

\begin{lemma}
(\cite{Ch25})  Let $G$ and $H$ be graphs such that $V(G) = V(H)$ and $Aut(G) = Aut(H)$.  Then every distinguishing labeling of $G$ is also a distinguishing labeling of $H$ and vice versa so $D(G) = D(H)$.  
\label{lemmasameaut1a}
\end{lemma}

One simple consequence of the lemma is that $D(G) = D(\overline{G})$. We extend this result to celled graphs.

\begin{lemma}
Let $G$ and $H$ be graphs such that $V(G) = V(H)$.    Let $\mathcal{P}$ be a partition of $V(G)$.   If $Aut(G, \mathcal{P}) =   Aut(H, \mathcal{P})$, then every $\mathcal{P}$-distinguishing labeling of $G$ is also a $\mathcal{P}$-distinguishing labeling of $H$ and vice versa so $D(G, \mathcal{P}) =  D(H, \mathcal{P})$.
 \label{lemmasameaut1b}
\end{lemma}

\begin{proof}
Let $\phi$ be a  $\mathcal{P}$-distinguishing labeling of $G$.  If $\phi$ is not $\mathcal{P}$-distinguishing for $H$, then there is a nontrivial automorphism $\pi$ in $ Aut(H, \mathcal{P})$ so that $\phi(v) = \phi(\pi(v))$ for each $v \in V(H)$.  But $V(G) = V(H)$ and $\pi \in Aut(G, \mathcal{P})$ too so this would mean that $\phi$ is not $\mathcal{P}$-distinguishing for $G$, a contradiction.   A similar argument proves the converse.  Thus,  a labeling is $\mathcal{P}$-distinguishing for $G$ if and only if it is  $\mathcal{P}$-distinguishing for $H$.  It follows that $D(G, \mathcal{P}) =  D(H, \mathcal{P})$.
\end{proof}

We now present a result from Arvind et al.~\cite{Arvind2017GraphIsomorphism} about the automorphisms of amenable graphs.  Recall that the global structure of amenable graphs is described using their anisotropic components.

\begin{lemma} (\cite{Arvind2017GraphIsomorphism}, Claim 17)  Let $G$ be an amenable graph with anisotropic components $A_1, A_2, \hdots, A_k$ in $C(G)$.  
For each $i$, let $G_i = G[\cup_{X \in A_i} X]$,  the  subgraph induced by the vertices in the cells of $A_i$.  
 Then $$Aut(G) = \prod_{i=1}^k Aut(G_i, \mathcal{P}_G).$$
That is, $Aut(G)$ is the direct product of $Aut(G_1, \mathcal{P}_G)$, $\hdots$, $Aut(G_k, \mathcal{P}_G)$.
\label{autprod}
\end{lemma}

Note that Claim 17 in \cite{Arvind2017GraphIsomorphism} simply states that $Aut(G)$ is the direct product of $Aut(G_1), \hdots, Aut(G_k)$.  Implicit in their discussion, however, is that when $\pi \in Aut(G_i)$,  $\pi$ is also is cell-preserving with respect to $\mathcal{P}_G$.

\begin{theorem}
\label{mainthm1}
Let $G$ be an amenable graph whose anisotropic components in $C(G)$ are $A_1, A_2, \hdots, A_k$.   
For each $i$, let $G_i = G[\cup_{X \in A_i} X]$.  
A labeling $\phi$ of $G$ is distinguishing if and only if $\phi$ when restricted to the vertices of $G_i$  is $\mathcal{P}_G$-distinguishing for $i =1, \hdots, k$.  Consequently, $$D(G) = \max \{D(G_i, \mathcal{P}_G), i = 1, \hdots, k \}.$$

\end{theorem}

\begin{proof}
Let $\phi$ be a labeling of $G$ and denote as $\phi_i$ the labeling of $G_i$ when $\phi$ is restricted to $V(G_i)$. If $\phi_i$ is not  $\mathcal{P}_G$-distinguishing for $G_i$ then there is some non-trivial automorphism $\tau_i \in Aut(G_i, \mathcal{P}_G)$ that preserves the colors of $\phi_i$.  Now, the identity map $Id_j$ of $V(G_j)$ is an element of $Aut(G_j, \mathcal{P}_G)$ for each $j$.  By Lemma \ref{autprod},  $\pi = (Id_1, \hdots, Id_{i-1}, \tau_i, Id_{i+1}, \hdots, Id_k)$  is a non-trivial automorphism of $G$. But $\pi$ also preserves the colors assigned by $\phi$ so it is a non-trivial automorphism of $(G, \phi)$.  It follows that $\phi$ is not a distinguishing labeling of $G$. 
We have shown that if $\phi$ is a distinguishing labeling of $G$ then $\phi_i$ is $\mathcal{P}_G$-distinguishing for $G_i$, for each $i$.

Conversely, assume $\phi_i$ is  $\mathcal{P}_G$-distinguishing for $G_i$ for each $i$.  By Lemma \ref{autprod}, every automorphism $\pi$ of $(G, \phi)$ can be expressed as $(\pi_1, \pi_2, \hdots, \pi_k)$ where $\pi_i \in Aut(G_i, \mathcal{P}_G)$ for each $i$.  But from the assumption, each $\pi_i$ is trivial.  Thus, $\pi$ has to be trivial so $\phi$ is a distinguishing labeling of $G$.

We have shown that $\phi$ is a distinguishing labeling of $G$ if and only if $\phi_i$ is is  $\mathcal{P}_G$-distinguishing for $G_i$, for each $i$.  If $\phi$ uses $D(G)$ colors, then $\phi_i$ uses at most $D(G)$ colors as well so   $\max \{D(G_i, \mathcal{P}_G), i = 1, \hdots, k\} \leq D(G)$. On the other hand, if we pick each $\phi_i$ so that it uses $D(G_i, \mathcal{P}_G)$ colors then $\phi$ uses at most $\max \{D(G_i, \mathcal{P}_G), i = 1, \hdots, k\}$ colors so $D(G) \leq \max \{D(G_i, \mathcal{P}_G), i = 1, \hdots, k\}$.   Hence, $D(G) = \max \{D(G_i, \mathcal{P}_G), i = 1, \hdots, k\}.$
\end{proof}

\subsection{Some Simplifications}

According to Theorems \ref{mainthm1},  when $G$ is an amenable graph, we can determine $D(G)$ by computing  $D(G_i, \mathcal{P}_G)$ for each $i$. 
The structure of $G_i$, however, appears to be complicated.  Hence, we will make a sequence of modifications to $G_i$ to create a new graph $J_i$ on the same vertex set so that $D(G_i, \mathcal{P}_G) = D(J_i, \mathcal{P}_G)$.   The types of modifications are described in the two lemmas below. We note that as we make changes to $G_i$, the underlying graph is no longer $G_i$ but some graph $F$ on the same vertex set.




\begin{lemma}
Let $G$ be an amenable graph, $A_i$  the $i$th anisotropic component of $G$ and $G_i = G[\cup_{X \in A_i} X]$.
 Let $F$ be a graph on $V(G_i)$ and let $X$ be a cell of $A_i$.  Denote as $F'$ the graph obtained from $F$ by replacing $F[X]$ with its complement.  
Then $Aut(F, \mathcal{P}_G) = Aut(F', \mathcal{P}_G)$. 
\label{lemmamodif1}
\end{lemma}


\begin{proof}
Let $\pi \in Aut(F, \mathcal{P}_G)$.   We will show that $\pi \in Aut(F', \mathcal{P}_G)$.  Consider a pair of vertices $\{u,v\}$ of $F'$.  If $u, v \in X$, then $\pi(u), \pi(v) \in X$ because $\pi$ is cell-preserving.  Now $uv \in E(F')$ if and only if $uv \not \in E(F)$ by construction.  Since $\pi$ is an automorphism of $F$, $uv \not \in E(F)$ if and only if $\pi(u)\pi(v) \not \in E(F)$. The latter is true if and only if $\pi(u)\pi(v)  \in E(F')$ by construction.  So we've shown that $uv \in E(F')$ if and only if $\pi(u) \pi(v) \in E(F')$.   Suppose  $u \not \in X$ or $v \not \in X$.  Then $\pi(u) \not \in X$ or $\pi(v) \not \in X$.  Thus, the adjacency/non-adjacency between $u$ and $v$ and between $\pi(u)$ and $\pi(v)$ were unaffected when $F[X]$ was replaced by its complement. 
 Since $\pi$ is an automorphism of $F$,  $uv$ is an edge of $F$ and $F'$ if and only if $\pi(u)\pi(v)$ is an edge of $F$ and $F'$.  We have shown that $Aut(F, \mathcal{P}_G) \subseteq  Aut(F', \mathcal{P}_G)$. 

Suppose we modify $F'$ by replacing $F'[X]$ with its complement.  Call the new graph $F''$.  Then the argument above implies that  $Aut(F', \mathcal{P}_G) \subseteq  Aut(F'', \mathcal{P}_G)$.  But the complement of $F'[X]$  is just $F[X]$ so $F''= F$.  Thus,  $Aut(F, \mathcal{P}_G) = Aut(F', \mathcal{P}_G)$.
\end{proof}



The proof of the next lemma is just like the one above so we omit it. 

\begin{lemma}
Let $G$ be an amenable graph, $A_i$  the $i$th anisotropic component of $G$ and $G_i = G[\cup_{X \in A_i} X]$. 
 Let $F$ be a graph on $V(G_i)$ and let $X$ and $Y$ be cells in $A_i$.  Denote as $F''$  the graph obtained from $F$ by replacing $F[X, Y]$ with its complement.  Then $Aut(F, \mathcal{P}_G) = Aut(F'', \mathcal{P}_G)$. 
\label{lemmamodif2}
\end{lemma}

  Consider the anisotropic component $A_i$.  By Lemma \ref{global}, $A_i$ is a tree.  If $A_i$ has a heterogeneous vertex, root $A_i$ at this vertex.  Otherwise, pick a cell whose size is the least among the cells in $A_i$ and root $A_i$ at this vertex.  We shall refer to the root of $A_i$ as $R_i$.   Here are the specific modifications we want to make on $G_i$.

{\bf (M1) For root $R_i$:} The cell $R_i$ is either the lone heterogeneous vertex in $A_i$ or one of the homogeneous vertices with the least size in $A_i$.  If $G_i[R_i]$ is a matching graph $rK_2$ or an empty graph, replace $G_i[X]$ with its complement. 


{\bf (M2) For every other cell $X$ of $A_i$:}  The cell $X$ is a homogeneous vertex.  If $G_i[X]$ is a complete graph, replace $G_i[X]$ with its complement.

{\bf (M3) For each pair of cells $X, Y$ of $A_i$:}  The pair is either an isotropic or anisotropic edge in $C(G)$.  If $XY$ is an isotropic edge (and therefore not present in $A_i$) and $G_i[X, Y]$ is a complete bipartite graph, replace it with an empty bipartite graph. If $XY$ is an anisotropic edge and $G_i[X,Y]$ is the complement of  $sK_{1,t}$, replace it with $sK_{1,t}$.


Call the resulting graph $J_i$.

 \begin{figure}

\begin{tikzpicture}[
  every node/.style={draw=none, rounded rectangle, minimum width=2cm, minimum height=0.8cm},
  level distance=1.5cm,
  sibling distance=2cm
  ]

  \node[fill=lightgray] {$R_i$ (5)}
    child {node[fill=lightgray] {$X_1$ (10)}
    	child {node[fill=lightgray] {$Y_1$ (30)}}
	child {node[fill=lightgray] {$Y_2$ (20)}}
	}
    child {node[fill=lightgray] {$X_2$ (15)}}
    child {node[fill=lightgray] {$X_3$ (5)}
        child {node[fill=lightgray] {$Z_1$ (15)}}
      };
\end{tikzpicture}
\hspace*{2em}
\begin{tikzpicture}[
  outer/.style={draw=none, fill=lightgray, rounded rectangle, minimum width=2.2cm, minimum height=0.75cm},
  inner/.style={circle, fill=black, inner sep=2pt},
  every path/.style={thick}
]

\node[outer] (root) at (0,0) {};
\node[outer] (child1) at (-2.5,-1.5) {};
\node[outer] (child2) at (0,-1.5) {};
\node[outer] (child3) at (2.5,-1.5) {};
\node[outer] (g1) at (-3.5,-3) {};
\node[outer] (g2) at (-1.5,-3) {};
\node[outer] (g3) at (2.5,-3) {};

\node[inner] (r1) at ($(root)+(0,0)$) {};

\node[inner] (c11) at ($(child1)+(-0.3,0)$) {};
\node[inner] (c12) at ($(child1)+(0.3,0)$) {};

\node[inner] (c21) at ($(child2)+(-0.6,0)$) {};
\node[inner] (c22) at ($(child2)+(0,0)$) {};
\node[inner] (c23) at ($(child2)+(0.6,0)$) {};

\node[inner] (c31) at ($(child3)+(0,0)$) {};

\node[inner] (g11) at ($(g1)+(-0.75,0)$) {};
\node[inner] (g12) at ($(g1)+(-0.45,0)$) {};
\node[inner] (g13) at ($(g1)+(-0.15,0)$) {};
\node[inner] (g14) at ($(g1)+(0.15,0)$) {};
\node[inner] (g15) at ($(g1)+(0.45,0)$) {};
\node[inner] (g16) at ($(g1)+(0.75,0)$) {};

\node[inner] (g21) at ($(g2)+(-0.6,0)$) {};
\node[inner] (g22) at ($(g2)+(-0.2,0)$) {};
\node[inner] (g23) at ($(g2)+(0.2,0)$) {};
\node[inner] (g24) at ($(g2)+(0.6,0)$) {};

\node[inner] (g31) at ($(g3)+(-0.6,0)$) {};
\node[inner] (g32) at ($(g3)+(0,0)$) {};
\node[inner] (g33) at ($(g3)+(0.6,0)$) {};

\foreach \ci in {c11, c12, c21, c22, c23, c31} {
   \draw (r1) -- (\ci);
 }
 \foreach \gi in {g11, g12, g13, g21, g22}{
   \draw (c11) -- (\gi);
  }

 \foreach \gi in {g14, g15, g16, g23, g24}{
   \draw (c12) -- (\gi);
  }

\foreach \gi in {g31, g32, g33}{
   \draw (c31) -- (\gi);
  }
\end{tikzpicture}
\caption{On the left is an anisotropic component $A_i$ rooted at $R_i$.  Each vertex of $A_i$ is a cell.  The numbers indicate the sizes of the cells.  The graph $G_i$ induced by the vertices in the cells of $A_i$ is not shown.  When $G_i$ is modified using steps (M1), (M2) and (M3), the result is the jellyfish graph $J_i$.  The right figure shows a leg of $J_i$ and highlights the fact that it is a celled graph. Its structure mimics $A_i$ but the branching of the vertices is based on the sizes of the parent and child cells.}
\label{fig:examplewithleg}

 \end{figure}
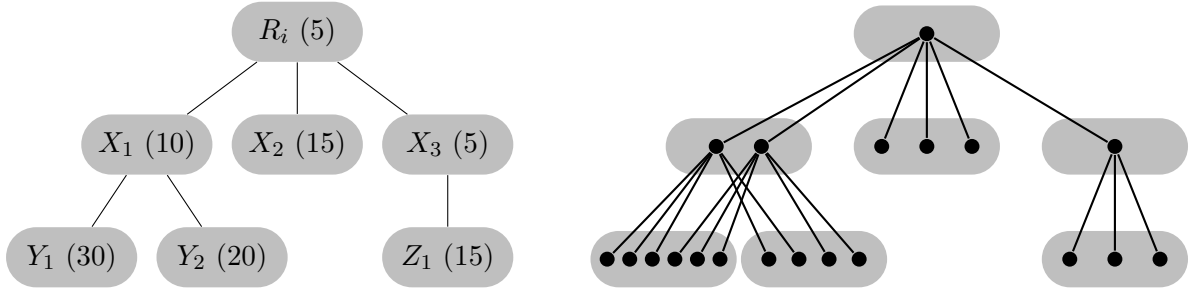

\begin{theorem}
Let $J_i$ be the graph obtained by performing the modifications (M1), (M2) and (M3) on $G_i$.  Then $Aut(G_i, \mathcal{P}_G) = Aut(J_i, \mathcal{P}_G)$.   Consequently, $D(G_i, \mathcal{P}_G) = D(J_i, \mathcal{P}_G)$.  
\label{thmconvert}
\end{theorem}

\begin{proof}
We note that $J_i$ is obtained by performing a sequence of modifications on $G_i$ and the modifications are of the type described in Lemmas \ref{lemmamodif1} and  \ref{lemmamodif2}.  It follows that the automorphisms of the celled graphs throughout the process stayed the same.  In particular, the first celled graph has the same automorphisms as the last celled graph; i.e., $Aut(G_i, \mathcal{P}_G) = Aut(J_i, \mathcal{P}_G)$. By Lemma \ref{lemmasameaut1b},  
 this implies that $D(G_i, \mathcal{P}_G) = D(J_i, \mathcal{P}_G)$. 
\end{proof}


\section{The Distinguishing Numbers of Amenable Graphs, Part  2}

Let us analyze the structure of $J_i$ next.  Recall that a jellyfish graph has a head $H$ that contains a Hamiltonian cycle  and rooted at every vertex of $H$ is a leg that is isomorphic to $L$, a rooted tree.  Additionally, no two legs share a vertex. In our next theorem, we will prove that $J_i$ is a jellyfish graph.  Moreover, its legs are {\it cell-isomorphic}; that is, for any two legs $L_v$ and $L_w$, there is an isomorphism $\pi$ so that $a$ and $\pi(a)$ belong to the same cell of $\mathcal{P}_G$ for any vertex $a$ of $L_v$.  See Figure \ref{fig:examplewithleg} for an example of what a leg of $J_i$ might look like. 



\begin{theorem}
Let $J_i$ be the graph obtained by performing the modifications (M1), (M2) and (M3) on $G_i$.  Then $J_i$ is a jellyfish graph and its legs are pairwise cell-isomorphic.  
\label{structjellyfish}
\end{theorem}

\begin{proof}
From Lemma \ref{local}, $G_i[R_i]$ is either a 5-cycle, an empty graph, a complete graph, a matching graph or its complement.  After  (M1) is applied, $J_i[R_i]$ is either a 5-cycle, a complete graph or the complement of a matching graph.  It is easy to check that in each case $J_i[R_i]$ has a Hamiltonian cycle.  We shall now refer to $J_i[R_i]$ as the {\it head} of $J_i$.  

To determine the legs of $J_i$, let us first consider its edges.  After applying (M1), (M2) and (M3), the graph $J_i[X]$ is  empty for any cell $X \neq R_i$ of $A_i$ and the bipartite graph $J_i[X,Y]$ is also empty for any two cells $X$ and $Y$ that are not adjacent in $A_i$ since $XY$ is an isotropic edge of $C(G)$.  Thus, if $xy$ is an edge of $J_i$,  then either (i) $x, y \in R_i$ or (ii) $x \in X$, $y \in Y$ and $XY$ is an edge of $A_i$.  

Recall that $A_i$ is a rooted tree.  By Lemma \ref{global}, if $X$ is the parent of $Y$, then $|X| \leq |Y|$. From (M3), $J_i[X,Y]$ has the structure of $sK_{1,t}$ with the centers of the stars in $X$ and the leaves in $Y$.  Since no two vertices in $X$ share a neighbor in $Y$, every $y \in Y$ has a unique neighbor in $X$, which we shall designate as the {\it parent of $y$}.  This parent-child relationship  induces a hierarchy on the vertices of $J_i$ which we will now take advantage of.





For each cell $X$ in $A_i$, let $T_X$ be the subtree of $A_i$ rooted at $X$.  For each $x \in X$, let $L_x$ be the subgraph induced by $x$ and the vertices that have $x$ as an ancestor.   We shall use induction on the height of $T_X$ to prove the following claim:

{\it Claim:}  For each $x \in X$, $L_x$ is a tree and for $x, x' \in X$, $L_x$ and $L_{x'}$ are disjoint and  cell-isomorphic.

{\it Proof of claim:} When $T_X$ has height $0$, $X$ is a leaf in $A_i$.  For $x, x' \in X$, $L_x$ and $L_{x'}$ just consists of $x$ and $x'$ respectively.  They are clearly disjoint trees and cell-isomorphic.  
Assume that the claim holds for cells $Y$ such that $T_Y$ has height less than $h$.  Let $T_X$ have height $h$ and let the children of $X$ in $A_i$ be $Y_1, Y_2, \hdots, Y_{\ell}$.  
Suppose $x \in X$  has $t_j$ neighbors in $Y_j$: $u_{j,1}, u_{j,2}, \hdots, u_{j, t_j}$.  By the induction hypothesis,  the set $\{L_{u_{j,k}},  k = 1, \hdots, t_j\}$ are pairwise disjoint trees for each $j$.  There are also no edges between  $L_{u_{j, k}}$ and $L_{u_{j',k'}}$ when $j \neq j'$ because it would mean that there is an edge between $Y_j$ and $Y_{j'}$ in $A_i$, contradicting the tree structure of $A_i$.   Since $x$ has edges only to the roots of the trees in $\bigcup_{j=1}^\ell \{L_{u_{j,k}},  k = 1, \hdots, t_j\}$,  $L_x$ is a tree.  

Let $x'$ be another vertex in $X$. We know that the neighbors of $x'$ in $Y_1, Y_2, \hdots, Y_{\ell}$ are distinct from those of $x$ even though the number of neighbors in each set is the same for both vertices.  Thus, $L_{x'}$ is also a tree that is disjoint from $L_x$.  Now, assume the $t_j$ neighbors of $x'$ in $Y_j$ are $w_{j,1}, w_{j,2}, \hdots, w_{j, t_j}$.  The trees $\{L_{u_{j,k}},  k = 1, \hdots, t_j\} \cup \{L_{w_{j,k}},  k = 1, \hdots, t_j\}$ are all pair-wise cell-isomorphic for each $j$ because of the induction hypothesis.  We construct the isomorphism $\pi$ from $L_x$ to $L_{x'}$ as follows: Let $\pi(x) = x'$.  For $j = 1, \hdots, \ell$, for $k= 1, \hdots, t_j$, let $\pi(u_{j,k}) = \pi_{j,k}(u_{j,k})$ where $\pi_{j,k}$ is the cell-isomorphism from $L_{u_{j,k}}$ to $L_{w_{j,k}}$.  Since $x$ and $x'$ belong to the same cell of $\mathcal{P}_G$, it follows that $\pi$ defines an cell-isomorphism from $L_x$ to $L_{x'}$.  $\Box$

We have shown that the claim is true.  For $v \in R_i$, let $L_v$ be the {\it leg} rooted at $v$.  From the claim, we know that the legs of $J_i$ are pairwise disjoint trees that are cell-isomorphic to each other.
\end{proof}

Now that we know that $J_i$ is a jellyfish graph, we would like to make use of the results  in Section \ref{sec:prelim} on the distinguishing labelings and numbers of jellyfish graphs.   
Lemma \ref{lemmajellyfish} can be extended to characterize the $\mathcal{P}_G$-distinguishing labelings of $J_i$.  When $\phi$ is a labeling of $J_i$, we just have to modify the notion of its projection, $\phi_{proj}$, because $J_i$ is now a celled graph.   Here is the updated definition:   for any $v, w \in V(H)$, let $\phi_{proj}(v) = \phi_{proj}(w)$ if and only if there is a cell-isomorphism from $L_v$ to $L_w$ that preserves the labels of $\phi$.


\begin{lemma}
Let $J_i$ be the jellyfish graph obtained by performing the steps (M1), (M2) and (M3) on $G_i$.    A labeling $\phi$ of $J_i$ is $\mathcal{P}_G$-distinguishing if and only if (i) $\phi_v$ is a $\mathcal{P}_G$-distinguishing labeling of $L_v$ for each $v \in V(H)$ and (ii) $\phi_{proj}$ is a $\mathcal{P}_G$-distinguishing labeling of $H$. 
\label{lemmajellyfish-cell}
\end{lemma}

The proof of Lemma \ref{lemmajellyfish2} is similar to that of Lemma \ref{lemmajellyfish} except that the automorphisms and isomorphisms are cell-preserving and the  labelings are destroying the cell-preserving automorphisms.

\begin{theorem}
Let $J_i$ be the jellyfish graph obtained by performing the steps (M1), (M2) and (M3) on $G_i$.  Suppose $H$ is its head and $L_v$ is one of its legs.  Then $$D(J_i, \mathcal{P}_G) = \min\{c: D((L_v, \mathcal{P}_G), c) \ge d^* \}$$ where $d^* = D(H)$. 
\label{distjellyfish-cell}
\end{theorem}

Again, the proof of Theorem \ref{distjellyfish-cell} follows that of Theorem \ref{distjellyfish}.  We note that $D(H, \mathcal{P}_G) = D(H)$ because $H = J_i[R_i]$; i.e., all the vertices in $H$ are part of the cell $R_i$ so all the automorphisms of $H$ are naturally cell-preserving.
\medskip



We now show how to compute $D((L_v, \mathcal{P}_G), c)$ directly from the anisotropic component $A_i$. Recall that $L_v$, a leg of $J_i$,  is a rooted tree.  
Let $x$ be one of its vertices and $L_x$ be the subtree rooted at $x$.  We want to  compute $D((L_x, \mathcal{P}_G), c)$ recursively using Theorem \ref{disttree}.  When $x$ is a leaf, 
$D((L_x, \mathcal{P}_G), c) = c$.   When $x$ is not a leaf, let $X$ be the cell that contains $x$.   The children of $x$ lie in the cells that are the children of $X$.   Let the latter be $Y_1, Y_2, \hdots, Y_\ell$.  
From Theorem \ref{disttree},  the set $\mathcal{T}$ contains the subtrees of $L_x$ rooted at the children of $x$. We then need to separate these subtrees into isomorphism classes.  But because we are dealing with celled trees, we need to separate the subtrees into cell-isomorphism classes -- i.e., two subtrees are in the same class if and only if they are cell-isomorphic.  It turns out that $A_i$ provides the cell-isomorphism classes  for free.


When $y, y'$ are children of $x$, the subtrees $L_y$ and $L_{y'}$ are cell-isomorphic if and only if $y$ and $y'$ are from the same cell $Y_j$.   The forward direction is obvious; the backward direction is from the proof of Theorem \ref{structjellyfish}.  Thus,  
$$\mathcal{T} = m_1 L_{y_1} \cup m_2 L_{y_2} \cup \cdots \cup m_\ell L_{y_\ell}$$
where $y_j \in Y_j$ and $m_j = |Y_j|/|X|$ for $j = 1, \hdots, \ell$.  The value of $m_j$ follows from the fact that the bipartite graph induced by $J_i[X, Y_j]$ is of the form $sK_{1,t}$ so every vertex in $X$ has $ |Y_j|/|X|$ neighbors in $Y_j$.   Applying Theorem \ref{disttree}, 
\begin{equation}
\label{newformula}
D((L_x, \mathcal{P}_G),c) = c \prod_{j = 1}^\ell \binom{D((L_{y_j}, \mathcal{P}_G),c)}{m_j}. 
\end{equation}
Given $A_i$ and the sizes of the cells in $A_i$, here now is the algorithm for computing $D((L_v, \mathcal{P}_G), c)$:  
Run a postorder traversal on $A_i$.  At each cell $X$, compute $D((L_x, \mathcal{P}_G), c)$ where $x \in X$.   Thus, when $X$ is a leaf,  set $D((L_x, \mathcal{P}_G), c) = c$, the base case.   When $X$ has $Y_1, Y_2, \hdots, Y_\ell$ as its children,  compute $m_j = |Y_j|/|X|$ for each $j$.  Then apply equation (\ref{newformula}) to obtain $D((L_x, \mathcal{P}_G),c)$.  Eventually, when the postorder traversal makes its way back to $R_i$, the result will be $D((L_v, \mathcal{P}_G), c)$. Since time spent at each vertex $X$ of $A_i$ is linear in the number of children of $X$, it follows that the postorder traversal runs in time linear in the size of $A_i$.



Consider the anisotropic component $A_i$ from Figure \ref{fig:examplewithleg}.  The computation for $ D((L_v, \mathcal{P}_G), 3)$ is shown in Figure \ref{fig:exampleAi}. 
When we run a post-order traversal on $A_i$, the cells are processed in the following order:  $Y_1, Y_2, X_1, X_2, Z_1, X_3, R_i$.  Let $c = 3$.  The numbers on the tree on the right equal $ D((L_x, \mathcal{P}_G), 3)$ for $x \in X$.  They were computed using the formula in (\ref{newformula}) with the base cases  set to $3$.


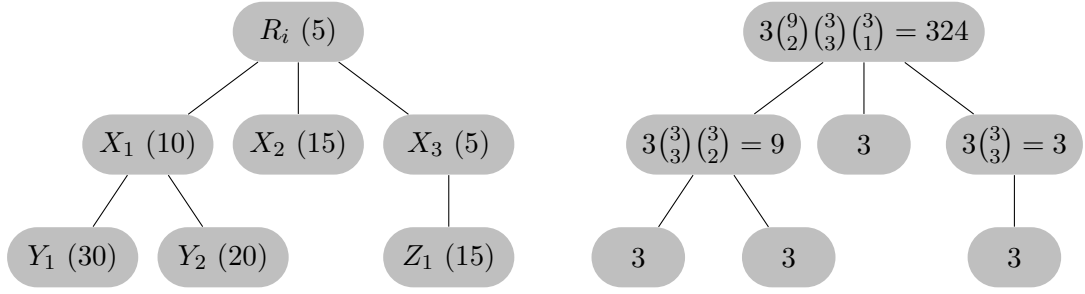
\begin{figure}

\centering
\begin{tikzpicture}[
  every node/.style={draw=none, rounded rectangle, minimum width=2cm, minimum height=0.8cm},
  level distance=1.5cm,
  sibling distance=2cm
  ]

  \node[fill=lightgray] {$R_i$ (5)}
    child {node[fill=lightgray] {$X_1$ (10)}
    	child {node[fill=lightgray] {$Y_1$ (30)}}
	child {node[fill=lightgray] {$Y_2$ (20)}}
	}
    child {node[fill=lightgray] {$X_2$ (15)}}
    child {node[fill=lightgray] {$X_3$ (5)}
        child {node[fill=lightgray] {$Z_1$ (15)}}
      };
\end{tikzpicture}
\hspace*{2 em}
\begin{tikzpicture}[
  every node/.style={draw=none, rounded rectangle, minimum width=1.5cm, minimum height=0.8cm},
  level distance=1.5cm,
  sibling distance=2cm
  ]

  \node[fill=lightgray] {$3 \binom{9}{2} \binom{3}{3} \binom{3}{1} = 324$}
    child {node[fill=lightgray] {$3 \binom{3}{3} \binom{3}{2} = 9$}
    	child {node[fill=lightgray] { $3$}}
	child {node[fill=lightgray] { $3$}}
	}
    child {node[fill=lightgray] {$3$}}
    child {node[fill=lightgray] {$3 \binom{3}{3} = 3$}
        child {node[fill=lightgray] {$3$}}
      };
\end{tikzpicture}
\caption{Continuing with the anisotropic component $A_i$ from Figure \ref{fig:examplewithleg}, the tree on the right shows how to compute $ D((L_v, \mathcal{P}_G), 3)$ bottom-up using equation (\ref{newformula}).}
\label{fig:exampleAi}
\end{figure}


\begin{theorem}
 Let $G$ be an amenable graph.  The distinguishing number of $G$ can be computed in $O((n+m) \log n)$ time where $n = |V(G)|$ and $m = |E(G)|$. 
 \label{maindistthm}
 \end{theorem}
 
 \begin{proof}
 To compute $D(G)$, we first have to identify the anisotropic components of $G$: $A_1, A_2, \hdots, A_k$.  Arvind et al.~described how to do this step  in $O((n+m) \log n)$ time in \cite{Arvind2017GraphIsomorphism} when they proved that amenable graphs can be recognized in the said amount of time.   Let $\mathcal{P}_G = \{X_1, X_2, \hdots, X_k\}$ be the stable partition of $G$.  They kept track of (i) $|X_i|$, the size of $X_i$, for each $i$  and (ii) $d_{i,j}$, the number of neighbors a vertex in $X_i$ has in $X_j$,  for every $i,j$.  If $d_{i,i} = 0$ or $|X_i| - 1$,  $X_i$ is a homogeneous cell; otherwise it is heterogeneous.   Similarly, if $d_{i,j} = 0$ or $|X_j|$, the edge $X_iX_j$ is isotropic; otherwise it is anisotropic.   Running BFS on $C(G)$ allowed them to identify the anisotropic components of $G$ and mark an appropriate root for each component.

 Next, let us compute $D(G_i, \mathcal{P}_G)$  for $i = 1, \hdots, k$.  By Theorem \ref{thmconvert},  $D(G_i, \mathcal{P}_G) = D(J_i, \mathcal{P}_G)$;  Theorem \ref{distjellyfish-cell} provides a formula for $D(J_i, \mathcal{P}_G)$.   There are two parts to this formula -- the computation of $D(H)$ and $D((L_v, \mathcal{P}_G),c)$ where $H$ is the head of $J_i$ and $L_v$ is one of its legs.  We already showed   that the latter can be obtained by running a post-order traversal on $A_i$ in time linear in the size of $A_i$. 
 So let us now focus on $D(H)$.  
 
 We know that $H = J_i[R_i]$.  The type of cell $R_i$ is and its size allows us to determine $D(H)$.  If $R_i$ is a homogeneous vertex, then $J_i[R_i]$ is a complete graph so $D(H) = |R_i|$.  Otherwise, $R_i$ is a heterogeneous vertex.  If $|R_i| = 5$, then $J_i[R_i]$ is a 5-cycle so $D(H) = 3$.  If $|R_i|$ is even, then $J_i[R_i]$ is the complement of a matching graph $rK_2$ with $r = |R_i|/2$.  But $D(\overline{rK_2}) = D(rK_2) = \min \{c: \binom{c}{2} \ge r\}$ and can be computed in $O(\sqrt{r})$ time (see appendix of \cite{Ch25}).  Thus, given $A_i$,  we can obtain both $D(H)$ and $D((L_v, \mathcal{P}_G),c)$ in time linear in the size of $G_i$.  Finally, we can use binary search over the range $[1, n]$ to find the minimum $c$ so that  $D((L_v, \mathcal{P}_G),c) \ge D(H)$.  Thus,  $D(G_i, \mathcal{P}_G)$ can be computed in $O(n_i \log n_i)$ time where $n_i = |V(G_i)|$.  
 
 By Theorem \ref{mainthm1}, $D(G) = \max\{ D(G_i, \mathcal{P}_G), i = 1, \hdots, k\}$.  Hence, from identifying the anisotropic components of $G$ to computing $D(G_i, \mathcal{P}_G)$ for $i = 1, \hdots, k$ and returning their maximum, $D(G)$ for amenable graphs can be obtained in $O((n+m) \log n)$ time. 
 \end{proof}


\section{The Fixing Numbers of Amenable Graphs}

    We now consider the problem of computing the fixing number of an amenable graph.  Most of the work entails applying the previous sections' results on distinguishing labelings and numbers  to fixing sets and numbers. 
    
When $S \subseteq V(G)$ and $\mathcal{P}$ is a partition of $V(G)$,  we say that $S$ is a {\it $\mathcal{P}$-fixing set} of $G$ if the only automorphism $\pi \in Aut(G, \mathcal{P})$ such that $\pi(s) = s$ for every $s \in S$ is the identity map.  The parameter $Fix(G, \mathcal{P})$ is the size of the smallest $\mathcal{P}$-fixing set of $G$.  The following should be obvious.

\begin{proposition}
Let $S \subseteq V(G)$ and $\mathcal{P}$ be a partition of $V(G)$.  Let $\phi_S$ be the vertex labeling that assigns distinct colors to the vertices in $S$ and another color (the ``null" color) to the vertices not in $S$.  Then $S$ is a $\mathcal{P}$-fixing set of $G$ if and only if $\phi_S$ is a $\mathcal{P}$-distinguishing labeling of $G$. 
\label{proptranslate2} 
\end{proposition} 

The counterparts of Lemmas \ref{lemmasameaut1a} and \ref{lemmasameaut1b}   for fixing numbers can be obtained by converting fixing sets to distinguishing labelings as described in Propositions \ref{proptranslate} and \ref{proptranslate2}.

\begin{lemma}
Let $G$ and $H$ be graphs with $V(G) = V(H)$ and $Aut(G) = Aut(H)$.  Then every fixing set of $G$ is a fixing set of $H$ and vice versa so $Fix(G) = Fix(H)$.
\label{lemmasameaut2a}
\end{lemma}

\begin{lemma}
Let $G$ and $H$ be graphs such that $V(G) = V(H)$.  Let $\mathcal{P}$ be a partition of $V(G)$.   If $Aut(G, \mathcal{P}) =   Aut(H, \mathcal{P})$, then every $\mathcal{P}$-fixing set of $G$ is also a $\mathcal{P}$-fixing set of $H$ and vice versa so $Fix(G, \mathcal{P}) =  Fix(H, \mathcal{P})$.
\label{lemmasameaut2b}
\end{lemma}

We now prove the first main result about the fixing sets of amenable graphs. 

\begin{theorem}
\label{mainthm1b}
Let $G$ be an amenable graph whose anisotropic components in $C(G)$ are $A_1, A_2, \hdots, A_k$.   
For each $i$, let $G_i = G[\cup_{X \in A_i} X]$,  the  subgraph induced by the vertices in the cells of $A_i$.  
Let $S \subseteq V(G)$ and $S_i = S \cap V(G_i)$ for $i = 1, \hdots, k$.   Then $S$ is a fixing set of $G$ if and only if $S_i$ is a $\mathcal{P}_G$-fixing set of $G_i$ for $i = 1, \hdots, k$.  Consequently, $$Fix(G) = \sum_{i=1}^k Fix(G_i, \mathcal{P}_G).$$
\end{theorem}

\begin{proof}
The proof of Theorem \ref{mainthm1} showed that a labeling $\phi$ of $G$ is distinguishing if and only if  $\phi_i$ is a  $\mathcal{P}_G$-distinguishing for $G_i$, for each $i$.  Translating the result to fixing sets, it implies that a subset $S$  is a fixing set of $G$ if and only if $S_i$ is a $\mathcal{P}_G$-fixing set of $G_i$ for each $i$.   Now $S$ is the disjoint union of the $S_i$'s  so $|S| = \sum_{i=1}^k |S_i|$.  To minimize $|S|$, we should minimize each $|S_i|$.  It follows that $Fix(G) = \sum_{i=1}^k Fix(G_i, \mathcal{P}_G)$.
\end{proof}

Thus, to compute $Fix(G)$, we need to determine $F(G_i, \mathcal{P}_G)$ for each $i$.  But as we  noted earlier,  $G_i$ seems to have a complicated structure so we modify $G_i$ to create $J_i$. 

\begin{lemma}
Let $J_i$ be the graph obtained by performing the modifications (M1), (M2) and (M3) on $G_i$.  Then  $Fix(G_i, \mathcal{P}_G) = Fix(J_i, \mathcal{P}_G)$. 
\label{lemmaconvert2}
\end{lemma}

\begin{proof}
From Theorem \ref{thmconvert}, we already know  that $Aut(G_i, \mathcal{P}_G) = Aut(J_i, \mathcal{P}_G)$.  By Lemma \ref{lemmasameaut2b},  $Fix(G_i, \mathcal{P}_G) = Fix(J_i, \mathcal{P}_G)$. 
\end{proof}

We have shown that $J_i$ is a jellyfish graph so we shall take advantage of our earlier results on the fixing sets and fixing numbers of jellyfish graphs.  Note that unlike $\phi_{proj}$, we do not have to modify the definition for $S_{proj}$.  
 
 \begin{lemma}
Let $J_i$ be the jellyfish graph obtained by performing the steps (M1), (M2) and (M3) on $G_i$.  A subset $S$ of $V(J_i)$ is a $\mathcal{P}_G$-fixing set if and only if  (i) $S_v$ is a $\mathcal{P}_G$-fixing set of $L_v$ for each $v \in V(H)$ and (ii) $S_{proj}$ is a $\mathcal{P}_G$-fixing set of $H$. 
\label{lemmajellyfish3}
\end{lemma}

We omit the proof as it can be obtained by converting $S$ into $\phi_S$ and applying Lemma \ref{lemmasameaut2b}.

\begin{figure}

\centering
\begin{tikzpicture}[
  every node/.style={draw=none, rounded rectangle, minimum width=2cm, minimum height=0.8cm},
  level distance=1.5cm,
  sibling distance=2cm
  ]

  \node[fill=lightgray] {$R_i$ (5)}
    child {node[fill=lightgray] {$X_1$ (10)}
    	child {node[fill=lightgray] {$Y_1$ (30)}}
	child {node[fill=lightgray] {$Y_2$ (20)}}
	}
    child {node[fill=lightgray] {$X_2$ (15)}}
    child {node[fill=lightgray] {$X_3$ (5)}
        child {node[fill=lightgray] {$Z_1$ (15)}}
      };
\end{tikzpicture}
\hspace*{2 em}
\begin{tikzpicture}[
  every node/.style={draw=none, rounded rectangle, minimum width=1.5cm, minimum height=0.8cm},
  level distance=1.5cm,
  sibling distance=2.75cm
  ]

  \node[fill=lightgray] {$2 \times 3 + (3-1) + 1 \times 2 = 10 $}
    child {node[fill=lightgray] {$(3-1) + (2-1) = 3$}
    	child {node[fill=lightgray] { $0$}}
	child {node[fill=lightgray] { $0$}}
	}
    child {node[fill=lightgray] {$0$}}
    child {node[fill=lightgray] {$ (3-1) =2$}
        child {node[fill=lightgray] {$0$}}
      };
\end{tikzpicture}
\caption{Using the anisotropic component in Figure \ref{fig:exampleAi}, the right tree shows the computation for $Fix((L_v, \mathcal{P}_G)$ using equation (\ref{newformula2}). }
\label{fig:exampleAi-fix}
\end{figure}
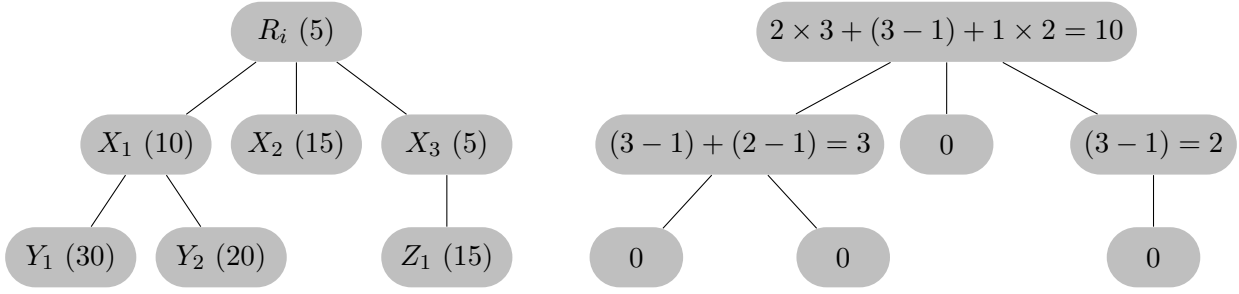


\begin{theorem}  
\label{thmjellyfish3}
Let $J_i$ be the jellyfish graph obtained by performing the steps (M1), (M2) and (M3) on $G_i$.   Suppose $H$ is its head and $L_v$ is one of its legs.  When 
$Fix(L_v, \mathcal{P}_G) = 0$, then $Fix(J_i, \mathcal{P}_G) = Fix(H)$; otherwise, $Fix(J_i, \mathcal{P}_G) = |V(H)| \times Fix(L_v, \mathcal{P}_G)$.
Equivalently,
$$ Fix(J_i, \mathcal{P}_G) =  Fix(H) \cdot \mathds{1}_{(Fix(L_v, \mathcal{P}) = 0)} + |V(H)| \times Fix(L_v, \mathcal{P}_G) \cdot \mathds{1}_{(Fix(L_v, \mathcal{P}_G) \neq 0)}. $$ 
\end{theorem}

The proof of Theorem \ref{thmjellyfish3} is similar to that of Theorem \ref{thmjellyfish2}.  Also, $Fix(H, \mathcal{P}_G) = Fix(H)$.  As with $D(L_v, \mathcal{P}_G)$, we will compute $Fix(L_v, \mathcal{P}_G)$ directly from $A_i$.   Assume $x$ is a vertex in $L_v$ and $L_x$ is the subtree rooted at $x$.  Let cell $X$ contain $x$. Our goal is to compute $Fix(L_x, \mathcal{P}_G)$ recursively using Theorem \ref{fixtree}.  When $x$ is a leaf, $Fix(L_x, \mathcal{P}_G) = 0$.  Otherwise, let the children of $X$ in $A_i$ be $Y_1, Y_2, \hdots, Y_{\ell}$.    Assume $\mathcal{T}$ contains the subtrees rooted at the children of $x$.   We need to separate the trees in $\mathcal{T}$ into cell-isomorphism classes.   We noted in the previous section that 
$$\mathcal{T} = m_1 L_{y_1} \cup m_2 L_{y_2} \cup \cdots \cup m_\ell L_{y_\ell}$$
where $y_j \in Y_j$ and $m_j = |Y_j|/|X|$ for $j = 1, \hdots, \ell$.   Applying Theorem \ref{fixtree} in this setting, 
\begin{equation}
\label{newformula2}
 Fix(L_x, \mathcal{P}_G) = \sum_{j=1}^\ell  \left[(m_j -1) \cdot \mathds{1}_{(Fix(L_{y_j}, \mathcal{P}_G) = 0)} + \left(m_j \times Fix(L_{y_j}, \mathcal{P}_G)\right) \cdot  \mathds{1}_{(Fix(T_{y_j}, \mathcal{P}_G) \neq 0)} \right]. 
\end{equation}
Running the postorder traversal on $A_i$ and updating the values of $Fix(L_x, \mathcal{P}_G)$ at each cell $X$ such that $x \in X$, we obtain $Fix(L_v, \mathcal{P}_G)$ at the root of $A_i$.   See Figure \ref{fig:exampleAi-fix} for an example.  Like $D(L_v, \mathcal{P}_G)$,  given $A_i$ and the sizes of its cells, the computation of $Fix(L_v, \mathcal{P}_G)$ can be done in time linear in the size of $A_i$.  

\begin{theorem}
Let $G$ be an amenable graph, the fixing number of $G$ can be computed in $O((n+m) \log n)$ time where $n = |V(G)|$ and $m = |E(G)|$. 
\label{mainfixthm}
\end{theorem}

\begin{proof}
In the proof of Theorem \ref{maindistthm}, we noted that the anisotropic components $A_1, A_2, \hdots, A_k$ of $G$ can be obtained in $O((n+m) \log n)$ time,  including the sizes of the cells in each $A_i$ and their individual types (homogeneous or heterogeneous).  The next step is to compute $Fix(G_i, \mathcal{P}_G)$ for each $i$.   We argued that it is easier to compute $Fix(J_i, \mathcal{P}_G)$ instead because $J_i$ is a jellyfish graph.  Theorem \ref{thmjellyfish3} provides the formula for $Fix(J_i, \mathcal{P}_G)$ which involves $Fix(L_v, \mathcal{P})$ and $Fix(H)$.  We already showed how to compute $Fix(L_v, \mathcal{P})$ in time linear in the size of $A_i$.  What is left is $Fix(H)$.  

Recall that $H = J_i[R_i]$.   If $R_i$ is homogeneous, $J_i[R_i]$ is a complete graph so $Fix(H) = |R_i| -1$.  If $R_i$ is heterogeneous with $|R_i| = 5$, then $J_i[R_i]$ is a $5$-cycle.  In this case, $Fix(H) = 2$.  If $|R_i|$ is even, then $J_i[R_i]$ is the complement of a matching graph $rK_2$ with $r = |R_i|/2$.  But $Fix(\overline{rK_2}) = F(rK_2)  = r$ from Proposition \ref{propcount2}.  Thus, knowing the type of $R_i$ and its size allows us to determine $Fix(H)$ in $O(1)$ time.   Combining this result with that for $Fix(L_v, \mathcal{P})$, we can compute $Fix(G_i, \mathcal{P}_G)$ in $O(n_i)$ time where $n_i = |V(G_i)|$, and all the $Fix(G_i, \mathcal{P}_G)$'s in $O(n)$ time.  

Hence, the bottleneck for computing $Fix(G)$ is the first part where we have to compute the anisotropic components of $G$ which takes $O((n+m) \log n)$ time.  The theorem follows.
\end{proof}

\section{Conclusion}

We combine all that we have learned about amenable graphs below.

\begin{theorem}
Let $G$ be an amenable graph.  There is a union of celled jellyfish graphs,  $\bigcup_{i=1}^k (J_i, \mathcal{P}_G)$, 
so that $Aut(G) =  \prod_{i=1}^k Aut(J_i, \mathcal{P}_G)$.   Furthermore,   $D(G) = \max \{ D(J_i, \mathcal{P}_G), i = 1, \hdots, k \}$ and $Fix(G) =  \sum_{i=1}^k Fix(G_i, \mathcal{P}_G)$.  
Moreover,  both $D(G)$ and $Fix(G)$ can be computed in  $O((n+m) \log n)$ time where $n = |V(G)|$ and $m = |E(G)|$. 
\end{theorem}

Given the importance of color refinement in graph isomorphism, there will likely be more research on amenable graphs.  We hope that conceptualizing an amenable graph as a union of celled jellyfish graphs will provide valuable insight.


\bibliography{dist-bib-meta}
\bibliographystyle{abbrv}

\end{document}